\documentclass [12pt]{article}
\usepackage{amssymb,amsmath,amsthm}
\usepackage{color}
\usepackage{graphicx}

\def\Im{{\rm Im}\,}

\usepackage[cp1250]{inputenc}
\usepackage[active]{srcltx}

\newtheorem{thm}{Theorem}[section]
\newtheorem{prop}[thm]{Proposition}

 \newtheorem{rem}[thm]{Remark}

\newcommand{\C}{\mathbb{C}}
\newcommand{\R}{\mathbb{R}}

\renewcommand{\H}{\mathbb{H}}

\long\def\comment#1{{}}

 \title{On the entropy for indeterminate moment problems}
 \author{Christian Berg}

\begin{document}

 \maketitle

\begin{abstract}   For an indeterminate Hamburger moment problem we consider  an infinite family of analytic densities solving the moment problem and we  prove that they all have finite (Shannon) entropy. These densities are either all bounded or all unbounded. The result is illustrated by the Al-Salam--Carlitz moment problem, where all the densities in the family are bounded.
\end{abstract}



\section{Introduction and main results} For a probability density $f$ on an interval $I$ the quantity
\begin{equation}\label{eq:ent}
H[f]:=-\int_I f(x)\log(f(x))\,dx 
\end{equation} 
is called the (Shannon) entropy of $f$, cf. \cite{K:K}. It is possible to construct examples,  where the entropy is either $-\infty$ or $\infty$, but if the density has second order moments then the entropy cannot be $\infty$ by the maximum entropy approach, see e.g. \cite[p. 115]{M:T}.

A Hamburger moment sequence is a sequence of numbers $(m_k)_{k\ge 0}$ for which there exists a positive measure $\mu$ on the real line $\R$ with moments of any order satisfying
\begin{equation}\label{eq:mom}
m_k=\int x^k\,d\mu(x),\quad k=0,1,\ldots.
\end{equation}
Since we will only be dealing with probability measures $\mu$, we assume that the moment sequence starts with $m_0=1$.

The moment sequence $(m_k)_{k\ge 0}$ is called determinate if there is only one probability measure on $\R$ satisfying \eqref{eq:mom}, and it is called indeterminate if there are more than one such measure, and in this case the set of measures satisfying \eqref{eq:mom}
 is an infinite convex set $V=V(m_k)$, which is compact in the weak as well as the vague  topology coinciding on $V$, see \cite{Ak},\cite{Sch}. The set $V$ is described by the so-called Nevanlinna parametrization from 1922 and using this, it was proved in \cite{B:C}
that there are many measures in $V$ with a $C^\infty$ density with respect to Lebesgue measure, and also many discrete  as well as many continuous singular measures in $V$. Here many means that the subsets of these three classes of measures are dense in $V$.

Let us describe the Nevanlinna parametrization, one of the gems of the moment problem.
The parameter set consists of the set $\mathcal N$ of Pick functions augmented by  a point at infinity to $\mathcal N^*:=\mathcal N\cup\{\infty\}$. Pick functions also appear in the literature under the names of Nevanlinna functions or Herglotz functions, and they are holomorphic functions $\varphi:\H\to \C$ in the upper halfplane $\H:=\{z\in \C\mid \Im(z)>0\}$ satisfying $\Im \varphi(z)\ge 0$ for $z\in \H$. They are usually extended to the lower half-plane by the definition $\varphi(z)=\overline{\varphi(\overline{z})},\Im z<0$. In general a Pick function cannot be extended holomorphically across parts of the real axis.

The Nevanlinna parametrization of $V$ is a homeomorphism $\varphi\mapsto \mu_\varphi$ of $\mathcal N^*$ onto $V$ given by
\begin{equation}\label{eq:Nev}
\int \frac{d\mu_\varphi(x)}{x-z}=-\frac{A(z)\varphi(z)-C(z)}{B(z)\varphi(z)-D(z)},\quad z\in \C\setminus\R,
\end{equation}
where $A,B,C,D$ are entire holomorphic functions defined entirely in terms of the moments.
One defines first the sequence $(p_n)_{n\ge 0}$ of orthonormal polynomials, where $p_n$ is uniquely determined  as a polynomial of degree $n$ with positive leading coefficient together with the orthogonality
\begin{equation}\label{eq:orth}
\int p_n(x)p_m(x)\,d\mu(x)=\delta_{n,m},\quad \mu\in V.
\end{equation}
There is also a classical determinant formula expressing $p_n$  in terms of the moments, see \cite[formula (1.4)]{Ak}.  Note that the integrals in \eqref{eq:orth} have the same value for all the measures in $V$ since they have the same moments. Afterwards one  defines the polynomials of the second kind
\begin{equation}\label{eq:sec}
q_n(z)=\int \frac{p_n(z)-p_n(x)}{z-x}\,d\mu(x), \quad z\in\C.
\end{equation}
Here $q_0=0$ and $q_n$ is a polynomial of degree $n-1$ when $n\ge 1$. Again the value of the right-hand side of \eqref{eq:sec} is independent of $\mu\in V$.

Finally one defines the
Nevanlinna functions of the indeterminate moment problem:
\begin{eqnarray}
A(z)&=&z\sum_{k=0}^\infty q_k(z)q_k(0)\label{eq:A}\\
B(z)&=&-1+z\sum_{k=0}^\infty p_k(z)q_k(0) \label{eq:B}\\
C(z)&=&1+z\sum_{k=0}^\infty q_k(z)p_k(0)\label{eq:C}\\
D(z)&=&z\sum_{k=0}^\infty p_k(z)p_k(0)\label{eq:D}.
\end{eqnarray}
These series only make sense in the indeterminate case, where they converge uniformly on  compact subsets of the complex plane. They therefore define entire holomorphic functions, which are real-valued for real $z$. Furthermore, they have infinitely many  zeros which are all real. The following remarkable relation holds:
\begin{equation}\label{eq:det}
A(z)D(z)-B(z)C(z)=1,\quad z\in\C.
\end{equation}

If a sequence of indeterminate moment sequences $(m_{j,k})_{k\ge 0}, j=1,2,\ldots$ is given, and if it converges to an indeterminate moment sequence $(m_k)_{k\ge 0}$ as $j\to \infty$, i.e.,
$m_{j,k}\to m_k$ for each $k$ when $j\to\infty$, then one can prove that under reasonable assumptions  the solutions $\mu_{j,\varphi}$ converge weakly to $\mu_\varphi$ as $j\to\infty$ for each fixed Pick function $\varphi$ from the parameter set $\mathcal N^*$, see \cite[Proposition 2.4.1]{B:V}.

The following analytic densities are available for any indeterminate Hamburger moment problem defined in terms of the functions $B, D$, see \cite{B:C} p. 105:
\begin{equation}\label{eq:big}
f_{t+i\gamma}(x)=\frac{\gamma}{\pi}\left[(tB(x)-D(x))^2+\gamma^2 B(x)^2\right]^{-1}, \quad x\in \R,
\end{equation}
where $t+i\gamma\in \H$, and this density is the solution in $V$ corresponding to the constant Pick function $z\mapsto t+i\gamma$. The special case $t=0,\gamma=1$ gives the very simple expression
\begin{equation}\label{eq:i}
f(x)=\frac{1}{\pi}\left( B(x)^2+ D(x)^2\right)^{-1},\quad x\in\R.  
\end{equation}

The papers \cite{M:T}, \cite{N:T1}, \cite{N:T2}, \cite{N:T3}, \cite{N:T4}, \cite{S:T:N} have been concerned about the existence of a uniquely determined density $g_{hmax}\in V$ which maximizes  the entropy $H[f]$ among the measures in $V$ having a density $f$ with respect to Lebesgue measure. In the proofs it is not excluded that the maximum entropy $H[g_{hmax}]$ can be $-\infty$, but since our main Theorem~\ref{thm:main} shows the existence of densities with finite entropy, the option 
$H[g_{hmax}]=-\infty$ is not possible. It should be emphasized that there are many more
measures in $V$ with densities than the family \eqref{eq:big}.

The lognormal density is a classical example of an indeterminate probability measure. In the paper \cite{N:T3} it is claimed that the lognormal density itself realizes the maximum entropy.

We shall here prove that all the densities \eqref{eq:big} have  finite entropy by an argument independent of the maximum entropy approach. 

\begin{thm}\label{thm:main} For an arbitrary indeterminate Hamburger moment problem
the densities $f_{t+i\gamma}$ have entropy $H[f_{t+i\gamma}]\in\R$ and the function
$t+i\gamma\mapsto H[f_{t+i\gamma}]$ is continuous from $\H$ to $\R$.
\end{thm} 
 
\begin{proof} An entire holomorphic function $f$ is said to be of minimal exponential type if for any $\varepsilon>0$ there exists a constant $K(\varepsilon)>0$ such that
\begin{equation}
|f(z)|\le K(\varepsilon)\exp(\varepsilon |z|),\quad z\in \C.
\end{equation}

A  theorem of Marcel Riesz, see \cite[Theorem 2.4.3]{Ak}, states that any of the four entire functions $A,B,C,D$ are of minimal exponential type. It follows that
$$
0<(tB(x)-D(x))^2+\gamma^2B(x)^2\le  c\exp(|x|), \quad x\in \R
$$
for a suitable constant $c>1$ depending on the functions $B,D$ and the constants $t,\gamma$. The first inequality holds because $B,D$ have no common zeros because of \eqref{eq:det}.
Similarly
$$
0<(tA(x)-C(x))^2+\gamma^2 A(x)^2\le  c\exp(|x|), \quad x\in \R,
$$
and it is certainly possible to use the same constant $c>1$ in both inequalities.
Note that by \eqref{eq:det}
$$
\gamma=(tA(x)-C(x))B(x)\gamma - A(x)\gamma(tB(x)-D(x)),
$$ 
so by the Cauchy-Schwarz inequality
$$
\gamma^2\le \left[(tA(x)-C(x))^2+\gamma^2A(x)^2\right]\left[(tB(x)-D(x))^2+\gamma^2B(x)^2\right].
$$ 
From this we get
\begin{eqnarray*}
\lefteqn{2\log\gamma - \log\left[(tA(x)-C(x))^2+\gamma^2A(x)^2\right]\le}\\ && \log\left[(tB(x)-D(x))^2+\gamma^2B(x)^2\right]\le \log c +|x|, 
\end{eqnarray*}
and hence
$$
-2|\log\gamma|-\log c-|x|\le
\log\left[(tB(x)-D(x))^2+\gamma^2B(x)^2\right]\le |x|+ \log c + 2|\log\gamma|
$$
showing that with $L:=\log c + 2 |\log\gamma|$ we have
\begin{equation}\label{eq:maj}
\left|\log\left[(tB(x)-D(x))^2+\gamma^2B(x)^2\right]\right|\le L+|x|.
\end{equation}
This shows that the entropy integral below is finite
\begin{eqnarray}\label{eq:entropy}
H[f_{t+i\gamma}]&=&\log(\pi/\gamma)+\frac{\gamma}{\pi}\int\frac{\log[(tB(x)-D(x))^2+\gamma^2B(x)^2]}{(tB(x)-D(x))^2+\gamma^2B(x)^2}dx,
\end{eqnarray}
as the density \eqref{eq:big} has moments of any order. 

It is also clear that  if $t+i\gamma$ belongs to a bounded subset $M$  of $\H$, then there exist constants $a,b>0$ such that
\begin{equation}\label{eq:maj2}
\left|\log\left[(tB(x)-D(x))^2+\gamma^2B(x)^2\right]\right|\le a+bx^2,\quad t+i\gamma\in M.
\end{equation} 

Let $(t_n+i\gamma_n)_{n\ge 1}$ denote a sequence from $\H$ converging to $t_0+i\gamma_0\in \H$. We shall prove that $H[f_{t_n+i\gamma_n}]$ converges to
$H[f_{t_0+i\gamma_0}]$.
For simplicity of writing we define
$$
h_n(x)=(t_nB(x)-D(x))^2+\gamma_n^2B(x)^2,\quad n=0,1,\ldots, x\in\R.
$$
We first note that since the Nevanlinna parametrization $\varphi\to \mu_\varphi$  is continuous, then we have weak convergence
\begin{equation}\label{eq:weak}
\lim_{n\to\infty} \int f_{t_n+i\gamma_n}(x)\,\psi(x)dx=
\int f_{t_0+i\gamma_0}(x)\psi(x)\,dx
\end{equation}
for any continuous and bounded function $\psi:\R\to\R$. By \eqref{eq:entropy} it suffices to prove that
\begin{equation}\label{eq:simply}
\lim_{n\to\infty}\int\frac{\log h_n(x)}{h_n(x)}dx=\int\frac{\log h_0(x)}{h_0(x)}dx.
\end{equation}
By \eqref{eq:maj2} there exist constants $a,b>0$ such that
\begin{equation}\label{eq:maj3}
|\log h_n(x)|\le a+bx^2,\quad n=0,1,\ldots.
\end{equation}
To a given $\varepsilon>0$ we choose a constant $K>0$ such that
\begin{equation}\label{eq:maj4}
\int_{\{|x|\ge K\}}\frac{a+bx^2}{h_0(x)}dx<\varepsilon,
\end{equation}
which is possible because the density $f_{t_0+i\gamma_0}$ has finite moments $m_k$ of all orders.
We next choose a continuous function $\varphi:\R\to [0,1]$, which is $1$ on the interval $[-K,K]$ and $0$ outside the interval $[-K-1,K+1]$. Letting $n\to\infty$ we get from the weak convergence \eqref{eq:weak}
\begin{eqnarray}\label{eq:cv}
\lefteqn{\int\frac{a+bx^2}{h_n(x)}(1-\varphi(x))\,dx=\frac{\pi}{\gamma_n}\left(a+bm_2
- \int f_{t_n+i\gamma_n}(x)(a+bx^2)\varphi(x)\,dx\right)} \nonumber\\
&\to& \frac{\pi}{\gamma_0}\left(a+bm_2-\int f_{t_0+i\gamma_0}(x)(a+bx^2)\varphi(x)\,dx\right)=\int\frac{a+bx^2}{h_0(x)}(1-\varphi(x))\,dx\nonumber\\
&\le& \int_{\{|x|\ge K\}}\frac{a+bx^2}{h_0(x)}dx<\varepsilon.
\end{eqnarray}
We finally have the following estimate involving three terms:
\begin{eqnarray*}
\left|\int\frac{\log h_0(x)}{h_0(x)}dx-\int\frac{\log h_n(x)}{h_n(x)}dx\right| \le
\left|\int_{\{|x|\ge K+1\}}\frac{\log h_0(x)}{h_0(x)}dx\right|\\+
\left|\int_{-K-1}^{K+1}\left(\frac{\log h_0(x)}{h_0(x)}-\frac{\log h_n(x)}{h_n(x)}\right)dx\right|
+\left|\int_{\{|x|\ge K+1\}}\frac{\log h_n(x)}{h_n(x)}dx\right|.
\end{eqnarray*}

The first term is $<\varepsilon$ by \eqref{eq:maj4}, the second is $<\varepsilon$ for $n$ sufficiently large, and the third is majorized by
$$
\int_{\{|x|\ge K+1\}}\frac{a+bx^2}{h_n(x)}dx\le \int\frac{a+bx^2}{h_n(x)}(1-\varphi(x))\,dx,
$$
which by \eqref{eq:cv} is $<\varepsilon$ for $n$ sufficiently large. This shows that
\eqref{eq:simply} holds.
\end{proof}

\begin{rem} {\rm It does not seem to be known if the densities \eqref{eq:big} are always  bounded or not. The unboundedness of the density \eqref{eq:i} can only happen if infinitely many large zeros of $B$ and $D$ are sufficiently close, but it seems difficult to construct such moment problems.

We show below that  all the densities \eqref{eq:big} are either bounded or they are all unbounded. A bounded density $f$ has necessarily entropy $H[f]>-\infty$ because if $f(x)\le C$ then
$$
H[f]=-\int \log(f(x)) f(x)\,dx\ge -\log C.
$$

In the next section we discuss the Al-Salam--Carlitz moment problem  and prove that all the densities \eqref{eq:big} are bounded for this moment problem.  }
\end{rem}

\begin{rem}{\rm
It is an open and interesting problem to find the Pick function $\varphi\in\mathcal N$ which corresponds to $g_{hmax}$ in the Nevanlinna parametrization. Even in the lognormal case this seems a difficult problem because the Stieltjes transform of the lognormal density
$$
E(z):=\frac{1}{\sqrt{2\pi}}\int_0^\infty \frac{\exp(-\log^2(x)/2)}{x(x-z)}dx,\quad z\in\H
$$
is not explicitly known.}
\end{rem}

\begin{prop}\label{thm:oneall} If one of the densities $f_{t+i\gamma}$ is bounded (resp. unbounded) then they are all bounded (resp. unbounded).
\end{prop}

\begin{proof} Assuming that $f_{t+i\gamma}$ is unbounded, there exists a sequence $(x_n)$ of real numbers such that $(tB(x_n)-D(x_n))^2+\gamma^2 B(x_n)^2\to 0$.
(The sequence $(x_n)$ is necessarily unbounded). However, this implies that $B(x_n)$ as well as $D(x_n)$ tend to 0, and then for all $s\in\R, \rho>0$
$$
(sB(x_n)-D(x_n))^2+\rho^2 B(x_n)^2\to 0,
$$
showing that the density $f_{s+i\rho}$ is unbounded.
\end{proof}

\section{The Al-Salam--Carlitz  moment problem}

This moment problem depends on two parameters $0<q<1$ and $a>0$ and is treated in \cite{B:V}. To describe it we recall the $q$-factorial notation  used. For a complex number $z$ we define
\begin{equation}
(z;q)_n:=\prod_{k=1}^n(1-zq^{k-1}), \quad n=0,1,\ldots,\infty,
\end{equation}
where $(z;q)_0=1$ as an empty product, and the infinite product $(z;q)_\infty$ is convergent because $q<1$. This is the standard notation used in \cite{G:R}. Since $q$  will be the same fixed number in this section, we have followed the notational simplification used in \cite{B:V}, namely $[z]_n:=(z;q)_n$. We restrict attention to the parameter values 
$0<q<1<a<1/q$, in which case the moment problem is indeterminate. It is in fact even indeterminate as a Stieltjes problem in the sense that there exist several measures with the same moments and supported on the half-line $[0,\infty)$, see \cite[p. 196]{B:V}.
The quantity $\alpha$ from \cite[formula (2.25)]{B:V}
$$
\alpha=\lim_{n\to\infty}\frac{p_n(0)}{q_n(0)}=\lim_{x\to-\infty}\frac{D(x)}{B(x)}
$$
is given by 
$$
\alpha=-\left(\sum_{n=0}^\infty \frac{[q]_n}{a^{n+1}}\right)^{-1}<0,
$$
cf. \cite[p. 200]{B:V}.

 We mention two discrete solutions, see \cite[Proposition 4.5.1]{B:V}:
\begin{equation}\label{eq:K}
\mu_K=[aq]_\infty\sum_{n=0}^\infty \frac{a^nq^{n^2}}{[aq]_n[q]_n}\delta_{(q^{-n}-1)},
\end{equation}

\begin{equation}\label{eq:F}
\mu_F=[q/a]_\infty\sum_{n=0}^\infty \frac{a^{-n}q^{n^2}}{[q/a]_n[q]_n}\delta_{(aq^{-n}-1)}.
\end{equation}
In these formulas we use the notation $\delta_p$ for the degenerate probability measure with mass 1 at the point $p$. We have used the notation $\mu_K,\mu_F$ for these measures  since they have been identified with the Krein and Friedrichs solutions to an indeterminate Stieltjes moment problem, see \cite{Pe}, \cite{B} and \cite[p. 178]{Sch}.
 Among the densities \eqref{eq:big} the following one-parameter family was found in the Al-Salam--Carlitz case, see \cite[Proposition 4.6.1]{B:V}:

\begin{equation}\label{eq:ACcont}
\nu_\rho(q,a)(x)=c(a)\frac{\rho}{[(1+x)/a]_\infty^2+\rho^2 [1+x]_\infty^2},\quad x\in\R, \rho>0,
\end{equation}   
where
\begin{equation}\label{eq:c(a)}
c(a)=\frac{a-1}{\pi a}[q]_\infty [aq]_\infty [q/a]_\infty.
\end{equation}
The densities $\nu_{\rho}(q,a)$ are related to the family \eqref{eq:big} in the following way:

For $t\in (\alpha,0)$ let $\gamma(t):=\sqrt{-t(t-\alpha)}$. Then $t+i\gamma(t)$ parametrizes the half-circle in $\H$ with diameter $[\alpha,0]$. Furthermore, $f_{t+i\gamma(t)}=\nu_{\rho}(q,a)$,
where
$$
\rho=\rho(t)=\frac{\gamma(t)}{t^2+\gamma^2(t)}\frac{(a-1)[q/a]_\infty}{a [q]_\infty [aq]_\infty},
$$
and $\rho$ is a bijection of $(\alpha,0)$ onto $(0,\infty)$.
The common moments of these measures are given by a complicated formula originally found by Al-Salam and Carlitz in 1965, see \cite[Section 4.9]{B:V}.

\begin{rem}{\rm If the parameters satisfy $0<q<a\le 1$ we still have an indeterminate Hamburger moment problem, but it is determinate as a Stieltjes problem. The measure $\mu_K$ in \eqref{eq:K} is the only solution to the Stieltjes problem and the measure $\mu_F$ in \eqref{eq:F} is a solution to the Hamburger moment problem with a negative mass point at $a-1$ when $q<a<1$.  The densities given in \eqref{eq:ACcont} are still solutions to the Hamburger problem when $q<a<1$, but 
the factor $a-1$ in \eqref{eq:c(a)} shall be replaced by $1-a$. See \cite{B:V} for details. The case  $a=1$  is also treated there. Note that $\mu_F=\mu_K$ when $a=1$.}
\end{rem}

All the densities $f_{t+i\gamma}$ are bounded as a consequence of Proposition~\ref{thm:oneall} and the following result:

\begin{thm}\label{thm:bd} Assume $1<a<1/q$. The function $\varphi(x):=[x]_\infty^2+[x/a]_\infty^2$ is bounded below on $\R$ by a positive constant $LB(a,q)$  given in \eqref{eq:LB}.
\end{thm}

\begin{proof} First notice that for $x\le \xi<1$ we have $[x]_\infty \ge [\xi]_\infty>0$, and hence
\begin{equation}\label{eq:1}
\varphi(x)\ge [\xi]_\infty^2,\quad x\le \xi<1.
\end{equation}
 We choose constants $\alpha,\beta$ such that
\begin{equation}
1<\alpha<a<\beta<1/q.
\end{equation} 
We first estimate $[x]_\infty^2$ from below when $x$ belongs to the interval $[\alpha q^{-n},\beta q^{-n}]$ for fixed $n=0,1,\ldots.$

To obtain the estimates we use that the parabola $(1-cx)^2$ has minimum at $x=1/c$, so the minimum of the parabola over an interval $[l,r]$ is achieved at the right endpoint $r$ of the interval when $r<1/c$ and at the left endpoint $l$ when $1/c<l$.
 
For $x\in[\alpha q^{-n}, \beta q^{-n}]$ we find
\begin{eqnarray*}
[x]_\infty^2&=&\left(\prod_{k=0}^{n-1}(1-xq^k)^2\right) (1-xq^n)^2\prod_{k=n+1}^\infty (1-xq^k)
^2\\
&\ge& \left(\prod_{k=0}^{n-1}(1-\alpha q^{k-n})^2\right)(1-\alpha)^2\prod_{k=n+1}^\infty(1-\beta q^{k-n})^2\\
&=& (\alpha-1)^2\left(\prod_{j=1}^n\left(\frac{\alpha}{q^j}\right)^2(1-q^j/\alpha)^2\right)\prod_{j=1}^\infty(1-\beta q^j)^2\\
&=& (\alpha-1)^2\alpha^{2n}q^{-n(n+1)}[q/\alpha]_n^2\,[\beta q]_\infty^2\\
&>&(\alpha-1)^2\alpha^{2n}q^{-n(n+1)}[q/\alpha]_\infty^2\,[\beta q]_\infty^2=:K_n.
\end{eqnarray*}
Note that $K_n$ is increasing in $n$ so $K_n\ge K_0$ and hence
\begin{equation}\label{eq:2}
\varphi(x)\ge K_0,\quad x\in\bigcup_{n=0}^\infty [\alpha q^{-n},\beta q^{-n}].
\end{equation} 

We next estimate $[x/a]_\infty^2$ from below for $x\in [\beta q^{-n+1},\alpha q^{-n}]$
for fixed $n=0,1,\ldots.$

We find
\begin{eqnarray*}
[x/a]_\infty^2&=&\left(\prod_{k=0}^{n-1}(1-(x/a)q^k)^2\right)(1-(x/a)q^n)^2\prod_{k=n+1}^\infty (1-(x/a)q^k)^2\\
&\ge&\left(\prod_{k=0}^{n-1}(1-(q\beta/a)q^{k-n})^2\right)(1-\alpha/a)^2\prod_{k=n+1}^\infty (1-(\alpha/a)q^{k-n})^2\\
&=& (1-\alpha/a)^2\left(\prod_{j=0}^{n-1} \left(\frac{\beta}{aq^j}\right)^{2}(1-(a/\beta)q^j)^2\right)\prod_{j=1}^\infty(1-(\alpha/a)q^j)^2\\
&=&(1-\alpha/a)^2(\beta/a)^{2n}q^{-n(n-1)}[a/\beta]_n^2\,[q\alpha/a]_\infty^2\\
&>&(1-\alpha/a)^2(\beta/a)^{2n}q^{-n(n-1)}[a/\beta]_\infty^2\,[q\alpha/a]_\infty^2=:L_n.
\end{eqnarray*}
Note that $L_n$ is increasing in $n$ so $L_n\ge L_0$ and hence
\begin{equation}\label{eq:3}
\varphi(x)\ge L_0,\quad x\in\bigcup_{n=0}^\infty [\beta q^{1-n},\alpha q^{-n}].
\end{equation}
From \eqref{eq:1} we have $\varphi(x)\ge [q\beta]_\infty^2$ for $x\le q\beta$.
Combining this with \eqref{eq:2} and \eqref{eq:3}, we see that $\varphi(x)$ is bounded below on $\R$ by the constant
$$
\varphi(x)\ge \min\{K_0,L_0,[q\beta]_\infty^2\}.
$$

To get  a constant which does not depend on the chosen values $\alpha,\beta$ we choose
$\alpha:=\sqrt{a}, \beta:=\sqrt{a/q}$ and get the following lower bound for $\varphi(x)$:
$$
\min\{(\sqrt{a}-1)^2[q/\sqrt{a}]_\infty^2[\sqrt{aq}]_\infty^2, (1/\sqrt{a}-1)^2 [q/\sqrt{a}]_\infty^2[\sqrt{aq}]_\infty^2, [\sqrt{aq}]_\infty^2\},
$$
which can be simplified to
\begin{equation}\label{eq:LB}
LB(a,q):=[\sqrt{aq}]_\infty^2[q/\sqrt{a}]_\infty^2\min\{(1/\sqrt{a}-1)^2, 1/[q/\sqrt{a}]_\infty^2\}.
\end{equation}
\end{proof}
 
\begin{rem}{\rm The estimates above also show that $\varphi(x)$ tends to infinity faster than any power of  $|x|$ when $|x|\to\infty$. Note that 
$$
\varphi(ax)=[x]_\infty^2+[x/(1/a)]_\infty^2\ge LB(a,q)
$$
and $q<1/a<1$, so the densities given in \eqref{eq:ACcont} are also bounded  in the case $q<a<1$.
}
\end{rem} 

\medskip 
We include a Maple plot of the density $\nu_1(q,a)(x)$ with the values $q=0.6, a=1.2$ as 
well as a list of the entropy for some values of $\rho$ calculated with Maple.

\medskip 
\begin{center}
\begin{tabular}{|l|l|} \hline
$\rho$ & $H[\nu_\rho(q,a)]$\\ \hline
0.01 & -2.1184$\ldots$\\ \hline
0.2 & 0.5216$\ldots$\\ \hline
0.5 & 0.9714$\ldots$\\ \hline
1 & 1.0617$\ldots$\\ \hline
2 & 0.9100$\ldots$\\ \hline
5 & 0.4000$\ldots$\\ \hline
10 & -0.1393$\ldots$\\ \hline
\end{tabular}
\end{center}

\begin{center}
\includegraphics[scale=0.4]{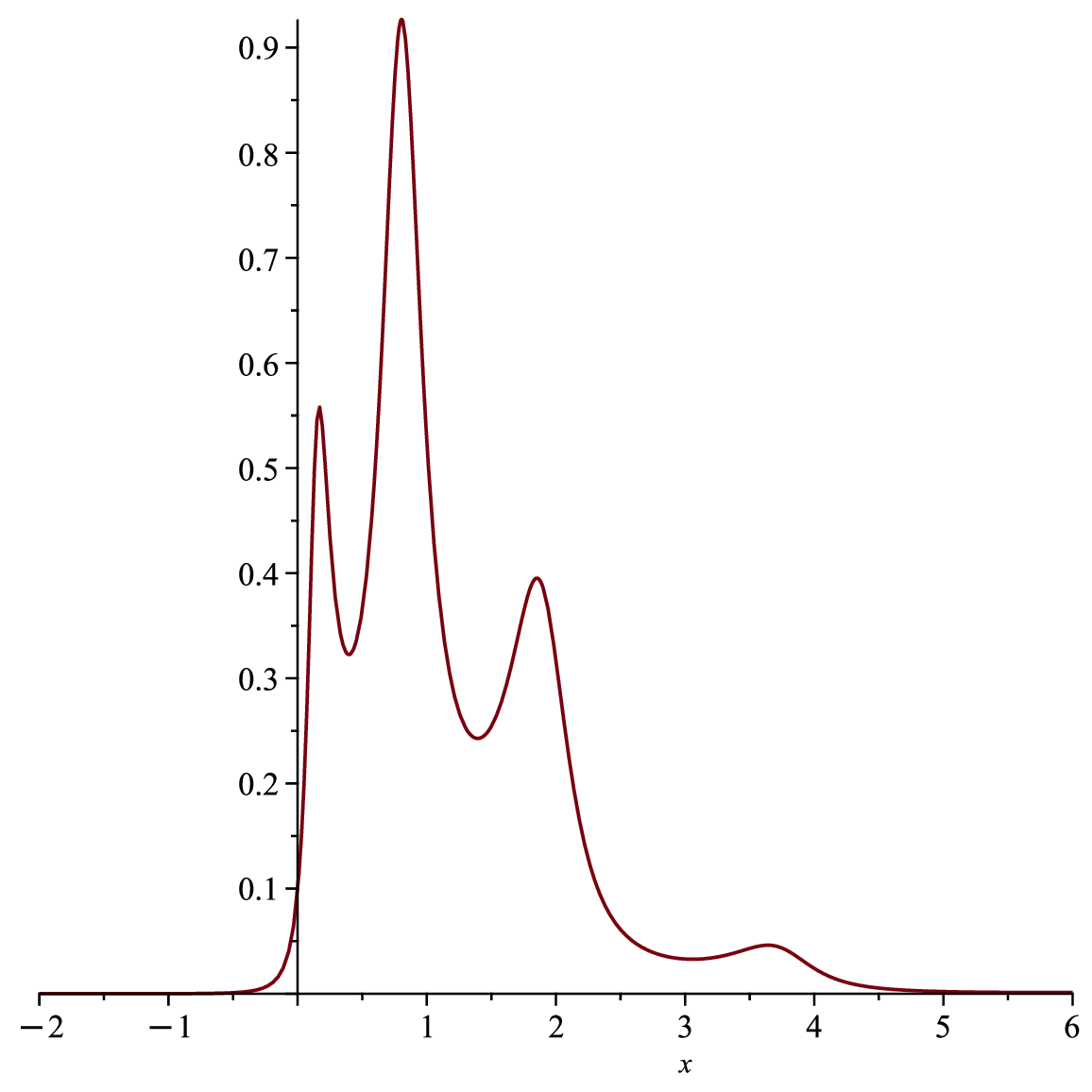}
\end{center}

{\bf Acknowledgment:} The author wishes to thank Henrik Laurberg Pedersen and Ryszard Szwarc for valuable comments during the preparation of the manuscript.

\noindent
Christian Berg\\
Department of Mathematical Sciences, University of Copenhagen\\
Universitetsparken 5, DK-2100 Copenhagen, Denmark\\
e-mail: {\tt{berg@math.ku.dk}}

\end{document}